\documentclass{birkjour}

\def\NZQ{\mathbb}               
\def\NN{{\NZQ N}}
\def\QQ{{\NZQ Q}}
\def\ZZ{{\NZQ Z}}
\def\RR{{\NZQ R}}

\def\AA{{\NZQ A}}

\def\frk{\frak}               

\def\Phi{{\frk n}}
\def\Phi{{\frk N}}
\def\MM{{\mathcal M}}

\def\ME{{\mathcal E}}

\def\opn#1#2{\def#1{\operatorname{#2}}} 
\opn\chara{char} \opn\length{\ell} \opn\pd{pd} \opn\rk{rk}
\opn\projdim{proj\,dim} \opn\injdim{inj\,dim} \opn\rank{rank}
\opn\depth{depth} \opn\grade{grade} \opn\height{height}
\opn\embdim{emb\,dim} \opn\codim{codim}

\opn\Tr{Tr} \opn\bigrank{big\,rank}
\opn\superheight{superheight}\opn\lcm{lcm}
\opn\trdeg{tr\,deg}
\opn\reg{reg} \opn\lreg{lreg} \opn\ini{in} \opn\lpd{lpd}
\opn\size{size}\opn\bigsize{bigsize}
\opn\cosize{cosize}\opn\bigcosize{bigcosize}
\opn\sdepth{sdepth}\opn\sreg{sreg}
\opn\link{link}\opn\fdepth{fdepth}
\opn\div{div} \opn\Div{Div} \opn\cl{cl} \opn\Cl{Cl}
\opn\Spec{Spec} \opn\Supp{Supp} \opn\supp{supp} \opn\Sing{Sing}
\opn\Ass{Ass} \opn\Min{Min}\opn\Mon{Mon} \opn\dstab{dstab} \opn\astab{astab}
\opn\Ann{Ann} \opn\Rad{Rad} \opn\Soc{Soc}
\opn\Im{Im} \opn\Ker{Ker} \opn\Coker{Coker} \opn\Am{Am}
\opn\Hom{Hom} \opn\Tor{Tor} \opn\Ext{Ext} \opn\End{End}
\opn\Aut{Aut} \opn\id{id}

\opn\nat{nat}
\opn\pff{pf}
\opn\Pf{Pf} \opn\GL{GL} \opn\SL{SL} \opn\mod{mod} \opn\ord{ord}
\opn\Gin{Gin} \opn\Hilb{Hilb}\opn\sort{sort}
\opn\aff{aff} \opn\con{conv} \opn\relint{relint} \opn\st{st}
\opn\lk{lk} \opn\cn{cn} \opn\core{core} \opn\vol{vol}
\opn\link{link} \opn\star{star}\opn\lex{lex} \opn\Gr{Gr}
\opn\gr{gr}

\def\pot#1#2{#1[\kern-0.28ex[#2]\kern-0.28ex]}

\opn\dirlim{\underrightarrow{\lim}}
\opn\inivlim{\underleftarrow{\lim}}
\let\tensor=\otimes

\def\Implies{\ifmmode\Longrightarrow \else
        \unskip${}\Longrightarrow{}$\ignorespaces\fi}
\def\implies{\ifmmode\Rightarrow \else
        \unskip${}\Rightarrow{}$\ignorespaces\fi}
\def\iff{\ifmmode\Longleftrightarrow \else
        \unskip${}\Longleftrightarrow{}$\ignorespaces\fi}

\let\:=\colon
\newtheorem{Theorem}{Theorem}[section]
\newtheorem{Lemma}[Theorem]{Lemma}

\newtheorem{Proposition}[Theorem]{Proposition}
\newtheorem{Remark}[Theorem]{Remark}

\newtheorem{Example}[Theorem]{Example}

\begin{document}

\title[Markov bases and generalized Lawrence liftings]{Markov bases and generalized Lawrence \\liftings}
\author[Charalambous]{Hara Charalambous}
\address{ Department of Mathematics\\ Aristotle University of Thessaloniki\\ Thessaloniki 54124\\ Greece}
\email{hara@math.auth.gr}

\author[Thoma]{Apostolos Thoma}
\address{ Department of Mathematics\\ University of Ioannina\\ Ioannina 45110\\ Greece }
\email{athoma@uoi.gr}

\author[Vladoiu]{Marius Vladoiu}
\address{ Faculty of Mathematics and Computer Science\\ University of Bucharest\\ Str. Academiei 14\\ Bucharest, RO-010014\\ Romania}
\email{vladoiu@gta.math.unibuc.ro}

\begin{abstract}
Minimal Markov bases of configurations of integer vectors
 correspond to minimal binomial generating sets of
 the assocciated lattice ideal. We give  necessary and
sufficient conditions for the elements of a minimal Markov basis 
 to be   (a) inside the
  universal Gr{\" o}bner basis
  and (b) inside the
  Graver basis. We study
properties  of  Markov bases of generalized Lawrence liftings for
arbitrary matrices $A\in\mathcal{M}_{m\times n}(\ZZ)$ and
$B\in\mathcal{M}_{p\times n}(\ZZ)$ and show that in cases of interest
the {\em complexity}
of any two Markov bases  is the same. 
\end{abstract}

\thanks{}
 \maketitle

\section{Introduction}

Let $A$ be an element of $\MM_{m\times n}(\ZZ)$, for some positive
integers $m,n$. The object of interest is the lattice
$\mathcal{L}(A):=\Ker_{\ZZ}(A)$.  A {\it Markov basis} $M$ of $A$ is a finite subset  of
$\mathcal{L}(A)$ such that whenever ${{{\bf w}}}, {\bf u}\in \NN^n$
and ${{{\bf w}}}- {\bf u}\in \mathcal{L}(A)$ (i.e. $A {{{\bf w}}}=A {\bf
u}$), there exists a subset $\{ {{\bf v}}_i: i=1,\ldots, s\}$ of $
{M}$ that {\it connects} ${{{\bf w}}}$ to ${\bf u}$. This means that
for   $1\leq p\leq s$, ${{{\bf w}}}+\sum^p_{i=1}{{{\bf v}}_i}\in\NN^n$
and
 ${{{\bf w}}}+\sum^s_{i=1}{{\bf v}}_i ={\bf u}$.
 A
Markov basis $M$ of $A$ gives rise to a generating set of the
lattice ideal
\[ I_{\mathcal{L}(A)}:= \langle x^{\bf u}-x^{{\bf v}}:\ A {\bf u}=A {{\bf v}} \rangle\ .\]
 Each ${{\bf u}}\in \ZZ^n$ can be uniquely
written as ${\bf u}= {{\bf u}}^+ -{\bf u}^-$ where ${\bf u}^+$, ${{\bf u}}^-
\in \NN^n$ are vectors with non-overlapping support. In the seminal work of Diaconis and Sturmfels in
\cite{DS}, it was shown   that   $M$ is a Markov basis of $A$
if and only if the set $\{x^{{\bf u}^+}-x^{{\bf u}^-}: \ {\bf u}\in M\}$ is a
generating set of $I_{\mathcal{L}(A)}$.
  A Markov basis $M$ of $A$  is {\it
minimal} if no subset of $M$ is a Markov basis of $A$. We say that  $\mathcal{L}(A)$ is
  {\em positive } if $\mathcal{L}(A) \cap \NN^{n}= \{
\bf 0\}$ and
  {\em non positive } if $\mathcal{L}(A) \cap \NN^{n}\neq \{
\bf 0\}$. When
$\mathcal{L}(A)$ is positive then the graded Nakayama Lemma
applies and all minimal Markov bases have the same cardinality.
When $\mathcal{L}(A)$ is non positive, it is possible to have
minimal Markov bases of $A$ of different cardinalities, see
\cite{CTV}. It is important to note that
the study of  non positive lattices   has important implications
in the study of positive ones, see for example  the proof of
\cite[Theorem 3]{SS}, \cite[Lemma 5]{HN} and \cite[Theorem 3.5]{HS}. The {\em universal Markov basis} of $A$ will be denoted by ${\MM}(A)$ and
is defined as the union of all minimal Markov bases of $A$ of minimal
cardinality, where we identify a vector ${\bf  u}$ with $-{\bf
u}$, see \cite{CTV, SS}.   The sublattice of $\mathcal{L}(A)$
generated by all elements of $\mathcal{L}(A)\cap \NN^n$ is called the
{\em pure} sublattice of $\mathcal{L}(A)$ and is important when considering minimal Markov bases of $A$, see \cite{CTV}. The pure sublattice of $\mathcal{L}(A)$ is zero exactly when $\mathcal{L}(A)$ is positive.

Let ${{\bf u}},{{\bf v}}, {{\bf w}}$ be non zero vectors in $\ZZ^n$.
If ${{\bf u}}={{\bf v}}+  {{{\bf w}}}$ we write  ${{\bf u}}={{\bf v}}+_c {\bf
w}$ to denote that this sum gives a {\em conformal decomposition}
of ${\bf u}$ i.e.~${\bf u}^+={\bf v}^+ + {{\bf w}}^+$ and ${\bf u}^-={\bf v}^- +
{{\bf w}}^-$. The set  consisting  of all non zero elements of $\mathcal{L}(A)
$ which have no conformal decomposition is denoted by
$\mathcal{G}(A)$ and is
called the {\em Graver basis} of $A$. When ${{\bf u}}\in \mathcal{G}(A)$ the binomial $x^{{\bf u}^+}-x^{{\bf u}^-}$ is called {\em primitive}.  $\mathcal{G}(A)$ is always a finite set, see  \cite{Gr,St}. 
In this paper we examine in detail  when an element of a minimal Markov basis belongs to  $\mathcal{G}(A)$.     We show in
  Theorem~\ref{comb_pure} that ${\MM}(A) \subset \mathcal{G}(A)$ holds
in
just two cases: when $\mathcal{L}(A)$ is positive and when
$\mathcal{L}(A)$ is pure of rank $1$. 
We point out that even though the inclusion for positive
lattices is well known, we could not locate its proof in the
literature, so we provide it here for completeness of the
exposition.

By $\mathcal U(A)$ we denote the universal Gr{\" o}bner basis of
$A$, i.e.~the set which consists of all vectors ${{\bf u}}\in \mathcal{L}(A)$ such
that $ x^{{\bf u}^+}-x^{{\bf u}^-}$ is part of a reduced Gr{\" o}bner basis of
$I_{\mathcal{L}(A)}$ for some term order on $\NN^n$.
 The inclusion $\mathcal U(A)\subset \mathcal G(A)$ always hold (see \cite[Lemma 4.6]{St}). In this paper we examine the relation between  ${\MM}(A)$ and $\mathcal{U}(A)$. In general
 $\MM(A)$ is not a subset of $\mathcal U(A)$ even when $\mathcal{L}(A)$ is
positive as Example~\ref{gr_univ_mark} shows.
 In Theorem~\ref{Grobner} we give
 a necessary and sufficient condition for $\MM(A)$ to be contained in $\mathcal U(A)$  when $\mathcal{L}(A)$ is positive.

In Section \ref{complexity}, for   $r\ge 2$,  $B\in \MM_{p\times n}(\ZZ)$,   we study   
the {\it generalized Lawrence lifting} $\Lambda
(A,B,r)$:
\[
\Lambda
(A,B,r)=\begin{array}{c} \overbrace{\quad\quad\quad \quad \quad\ \ }^{r-\textrm{times}} \\
 \left( \begin{array}{cccc}
\ A\ & 0 &   & 0 \\
0 &\ A\ &  & 0 \\
 & & \ddots &  \\
0 & 0 &  &\ A\ \\
B & B & \cdots & B
\end{array} \right)\end{array} \ .\]
When $B=I_n$ one gets the usual $r$--th {\it Lawrence lifting}
$A^{(r)}$, see \cite{SS}. Such liftings were used to prove for example the finiteness of the Graver basis of $A$ and are
connected to hierarchical models in Algebraic Statistics, 
see \cite{SS, HS}. We denote the columns of $A$ by ${\bf a}_1,
\dots, {\bf a}_n$ and the columns of $B$ by   ${\bf b}_1, \dots, {\bf b}_n$.  The $(rm+p)\times rn$ matrix $\Lambda(A, B, r)$ has columns the vectors
$$\{{\bf a}_i\tensor {\bf e}_j\oplus {\bf b}_i:\ 1\leq i\leq n, 1\leq j\leq r   \},$$
where ${\bf e}_1,\ldots,{\bf e}_n$ represents the canonical basis of $\ZZ^n$.
 Note that $\mathcal{L} ( \Lambda(A, B, r))$ is a sublattice of
$\ZZ^{rn}$.  Let $C \in \mathcal{L} ( \Lambda(A, B, r))$.
We can assign to  $C$ an $r\times n$ matrix   $\mathcal{C}$ such that  $\mathcal{C}_{i,j}=C_{(i-1)n+j}$.
Each row of $\mathcal{C}$ corresponds to an element of $\mathcal{L} (A)$ and the sum of the rows of $\mathcal{C}$
corresponds to an element in $\mathcal{L} (B)$. The number of nonzero rows of $\mathcal{C}$ is the
{\em type} of $C$. The {\em complexity} of any subset of  $\Lambda(A, B,
r)$ is the largest
type of any vector in that set. 
 
For $r\ge 2$ we consider the Graver basis of $\Lambda (A,B,r)$, $\mathcal{G}(\Lambda (A,B,r))$.
We let the {\em Graver complexity} of $(A,B)$  be  
the supremum over $r$ of   the complexities of $\mathcal{G}(\Lambda (A,B,r))$   and denote it  by $g(A,B)$.
By  \cite[Theorem 3.5]{HS},  $g(A,B)$ is finite and equals the maximum $1$-norm of the elements in the Graver
basis of the matrix $B\cdot \mathcal Gr(A)$, where $\mathcal Gr(A)$ is
the matrix whose columns are   the elements of  $\mathcal{G}(A)$.  
In the literature there are two definitions of  Markov complexity of $(A,B)$. The first   introduced in \cite{SS}  
defines the Markov complexity  of $(A,B)$ as the smallest integer $m$ such that there exists a Markov basis  of
$\Lambda (A,B,r)$ of type less than or equal to $m$ for any $r\ge 2$.  It is always finite and bounded by  $g(A,B)$.
The second definition given in \cite{HS}  defines the Markov complexity of  $(A,B)$
  as  the largest type of any element in the universal Markov basis of $\Lambda (A,B,r)$ as $r$ varies.
  The main result of   Section 3 is Theorem \ref{same_complexity}. 
It states that   when $\mathcal{L} ( \Lambda(A, B, r))$  is positive, 
 all minimal Markov bases of $\Lambda(A, B, r)$  
 have the same complexity. This is computationally essential, since to compute the complexity of any minimal Markov basis of $\Lambda(A, B, r)$ one can
 start with any monomial order in $\NN^n$, compute the reduced Gr{\" o}bner basis of $\Lambda(A, B, r)$, 
 eliminate extraneous elements to obtain a minimal Markov basisof $\Lambda(A, B, r)$ and then read the largest type of the remaining elements. 
We note that in general $\mathcal{L} ( \Lambda(A, B, r))$ is   
positive for  some  $r\ge 2$ if and only if $\mathcal{L} ( \Lambda(A, B, r))$ is   positive for 
all  $r\ge 2$
 and to decide whether this holds it suffices to check whether the lattice
 $\Ker_{\ZZ}(A)\cap \Ker_{\ZZ}(B)$ is   positive (see Lemma~\ref{kernel}). 
When $\Ker_{\ZZ}(A)\cap \Ker_{\ZZ}(B)$  is non positive then Theorem~\ref{prim_lem} shows that  
the Markov complexities of $\Lambda (A,B,r)$ cannot be bounded as $r$ varies.
 In  Example \ref{final_example} we give matrices $A, B$ such that for all $r\ge 2$, $\mathcal{L} ( \Lambda(A, B, r))$  
 has a minimal  Markov basis of  complexity $k$  for every $2\le k\le r$. 
 Finally, in  Remark \ref{markov-complexity} we discuss implications of our work to the two  notions of Markov complexity.

\section{Universal Markov and Gr{\" o}bner bases}

For simplicity of notation we write
  $\mathcal{L}$ for $\mathcal{L}(A)$,  ${\MM}$ for $\MM (A)$,
  $ \mathcal{G}$ for $\mathcal{G}(A)$, $ \mathcal{U}$ for
  $ \mathcal{U}(A)$ and
 $\mathcal{L}_{pure}$ for the pure  sublattice of $ \mathcal{L}$ generated by the elements in $\mathcal{L}\cap \NN^n$.
\begin{Theorem}\label{from_ctv}{\rm \cite[Theorem 4.18]{CTV}}
If i) $\rank(\mathcal{L}_{pure })>1$ or ii) $\rank(\mathcal{L}_{pure })=1$ and $\mathcal{L}\neq \mathcal{L}_{pure}$
then 
${\MM} $ is infinite.
\end{Theorem}
%
%
%
Next we consider the {\em fibers} $\mathcal{F}_{{{\bf u}}}$ of
$I_{\mathcal{L}}$ for any   ${{\bf u}}\in  \mathcal{L}$. We let
$\mathcal{F}_{{{\bf u}}}:= \{{\bf t}\in\NN^n : {{\bf u}}^+-{\bf t}\in
\mathcal{L}\}$. We note that if $\mathcal{L}$ is positive then
$\mathcal{F}_{{{\bf u}}}$ is a finite set. We construct a graph
${G}_{{\bf u}}$ with vertices   the elements of $\mathcal{F}_{{\bf
u}}$. Two vertices ${{\bf w}}_1$, ${{\bf w}}_2$ are joined by an edge if there is 
an index $i$ such that $i$-th component of ${{\bf w}}_1$ and ${{\bf w}}_2$ are nonzero. 
Thus $\bf w_1$, $\bf w_2$ are joined by an edge  if and only if   $({\bf w_1}-{\bf w_2})^+$ is componentwise smaller
than ${{\bf w}}_1$, meaning that at least one component
of their difference  is strictly positive. The following necessary condition for ${{\bf u}}\in
\mathcal{L}$ to be in ${\MM}$ was observed in \cite[Theorem
2.7]{CKT} and \cite[Theorem 1.3.2]{DSS} when $\mathcal{L}$ is
positive.

\begin{Theorem}\label{crit_min}
 If $\mathcal{L}$ is positive then ${\bf u}$ is in the universal Markov basis of $A$ if and only if ${\bf u}^+$ and
${\bf u}^-$ belong to different connected components of $G_{{\bf u}}$.
\end{Theorem}

$\mathcal{F}_{{{\bf u}}}$  is called a {\em Markov fiber} when there
exists an element ${{\bf v}}$ in the universal    Markov basis of
$A$ such that ${{\bf v}}^+\in\mathcal{F}_{{{\bf u}}}$. The {\em Markov
polyhedra} of $\mathcal{F}_{{{\bf u}}}$ are the convex hulls of the elements of the 
connected components  of   ${G}_{{\bf u}}$. When $\mathcal{L}$ is
positive, the  Markov polyhedra of $\mathcal F_{{\bf u}}$ are actually
polytopes. For ${{\bf u}}\in \mathcal{L}$ we let $P[{{\bf u}}]$ be
the convex hull of all elements of $G_{{\bf u}}$. When
${{\bf u}}\in \mathcal{M}$ the vertices of $P[{{\bf u}}]$ are vertices
of the Markov polyhedra of $\mathcal{F}_{{{\bf u}}}$ and we will
see that the vertices of the Markov polyhedra  are decisive
 in determining whether ${{\bf u}}$ belongs to the universal Gr{\" o}bner basis of $A$.  The criterion of
Theorem \ref{crit_min} is used in the proof of the next theorem.

\begin{Theorem}\label{comb_pure}
The universal Markov basis of $A$ is a subset of the Graver basis
of $A$ if and only if one of the following two conditions hold: i) $\mathcal{L}$ is positive or ii)
$\mathcal{L}=\mathcal{L}_{pure}$ and $\rank\mathcal{L}=1$. 
\end{Theorem}
\begin{proof} Assume first that $\MM \subset\mathcal{G} $.  Theorem~\ref{from_ctv} says
that if $\rank(\mathcal{L}_{pure })>1$ or $\rank(\mathcal{L}_{pure })=1$ and $\mathcal{L}\neq \mathcal{L}_{pure}$
then the universal Markov basis ${\MM}$ of $A$ is infinite. Since the Graver basis of $A$ is  finite the desired conclusion follows.
For the converse, assume first that $\mathcal{L}$ is non positive. By Theorem~\ref{from_ctv} we are in  the case where $\rank(\mathcal{L}_{pure})=1$ and $\mathcal{L}=\mathcal{L}_{pure}$. We let  ${\bf 0}\neq {{{\bf w}}}\in\NN^n$ be such that $\mathcal{L}=\langle {{{\bf w}}}\rangle$. It is immediate that ${{{\bf w}}}\in \mathcal{G} $ and thus ${\MM} = \mathcal{G} $. Next we examine the case where  $\mathcal{L}$ is positive. We will show that if ${{\bf u}}\in \mathcal{L}$, ${{\bf u}}\notin \mathcal{G} $ then ${{\bf u}} \notin{\MM} $. Since ${{\bf u}} \notin \mathcal{G} $ there
exist nonzero vectors ${{\bf v},{{\bf w}}}\in\mathcal{L}$ such that ${{\bf u}}={{\bf v}}+_c {{{\bf w}}}$. Thus
${\bf u}^+={\bf v}^++ {{\bf w}}^+$ and ${\bf u}^-={\bf v}^-+ {{\bf w}}^-$.
It follows that ${\bf u}^+$, ${\bf u}^-$ and ${\bf u}^+-{\bf v}={{\bf w}}^++{\bf v}^-$ are all in
  $ \mathcal F_{{\bf u}}$.  Next we show that
 ${\bf v}^+$ is nonzero. Indeed suppose not. Since   ${{\bf v}}^- = {{\bf v}} -{{\bf v}}^+={{\bf v}} \in \mathcal{L}\cap\NN^n=\{\bf 0\}$,
it follows that  ${\bf v}=\bf 0$, a contradiction. Similarly ${{\bf w}}^+, {\bf v}^-, {{\bf w}}^-$ are nonzero.
Thus  in the fiber  $ \mathcal F_{{\bf u}}$, the elements ${\bf u}^+$,  ${{\bf w}}^+ +{\bf v}^-$ are connected by an edge because $({{\bf u}}^+-({{{\bf w}}}^++{{\bf v}}^-))^+={{\bf v}}^+$, which is smaller than ${{\bf u}}^+$. Similarly
 ${\bf u}^-$, ${{\bf w}}^++{\bf v}^-$  are connected by an edge because $(({{{\bf w}}}^++{{\bf v}}^-)-{{\bf u}}^-)^+={{{\bf w}}}^+$, which is smaller than ${{{\bf w}}}^++{{\bf v}}^-$.
It follows that ${\bf u}^+$, ${\bf u}^-$
belong to the same connected component of $G_{{\bf u}}$ and thus ${{\bf u}}\notin \MM $.
\end{proof}

\begin{Remark}{\rm Let ${\ME} $  be the  union of all minimal Markov bases of
$A$, not necessarily of minimal cardinality,  where we identify a vector ${\bf  u}$ with $-{{\bf u}}$.  Note
that ${\MM} \subset {\ME} $.
Next we show that $\ME$ is a subset of $\mathcal{G} $ if and only if $\mathcal{L}$ is   a positive lattice. 
Indeed when $\mathcal{L}$ is   a positive lattice then $ {\ME}=\MM\subset \mathcal{G} $ as pointed out in the introduction (see also \cite{CTV}). 
Suppose now that 
  $\mathcal{L}$ is a non positive lattice and consider the only case left unanswered by Theorem~\ref{comb_pure}: namely consider the case 
where $\mathcal{L}=\mathcal{L}_{pure}$ of rank $1$. This means that $\mathcal{L}=\langle {{{\bf w}}}\rangle$ for some ${\bf 0}\neq {{{\bf w}}}\in\NN^n$. It is easy to see that if  $k,l\ge 2$ are any two relatively prime integers 
then
$\{k{{{\bf w}}},l{{{\bf w}}}\}$ is a minimal Markov basis of $A$. 
Thus $k{{{\bf w}}}\in {\ME}$ for every $k$ and therefore ${\ME}$ is infinite and  cannot be a subset 
of  the Graver
basis of $A$, which is equal to $\{{\bf w}\}$. }\end{Remark}

Next we examine the relation between the universal Markov basis of $A$ 
and the universal Gr{\" o}bner basis of $A$.   We recall the following characterization  given  in \cite{St}.

\begin{Theorem}\label{univ_grobner}\cite[Theorem 7.8]{St} If $\mathcal{L}$ is positive and ${{\bf u}}\in \mathcal{L}$
   then ${{\bf u}}$ is in the
universal Gr{\" o}bner basis of $A$ if and only if the greatest common divisor of the coordinates of
  ${{\bf u}}$ is one and the line segment $[{{\bf u}}^+,{{\bf u}}^-]$ is an edge of the polytope $P[{{\bf u}}]$.
\end{Theorem}

For ${{\bf u}}\in \RR^n$, we let
$\supp ({{\bf u}}):=\{ i:\ u_i\neq0\}$. For
$X\subset \RR^n$,  we let
\[ \supp(X):=\bigcup_{{{\bf u}}\in X}\supp ({{\bf u}})\ .\]
We note that if ${{\bf u}} \in \mathcal{M}$ and $\mathcal{L}$ is
positive then it is not hard to prove that the supports of
different connected components of $  G_{{{\bf u}}}$ are disjoint. Hence the Markov polytopes of $\mathcal F_{{\bf u}}$ are disjoint.

\begin{Lemma}\label{markov_groebner_positive} Let $\mathcal{L}$ be
positive. An element   ${\bf u}$  of the universal Markov basis of $A$
belongs to the universal Gr{\" o}bner basis of $A$ if and only if
${{\bf u}}^+$ and ${{\bf u}}^-$ are vertices of two different Markov polytopes.
\end{Lemma}

\begin{proof}  Suppose that ${\bf u}\in \mathcal{U}$. Since
${\bf u}\in \MM$, it follows by Theorem \ref{crit_min} that ${{\bf u}}^+$
and ${{\bf u}}^-$ are elements  of two different Markov polytopes of
$F_{{\bf u}}$.    Since $[{{\bf u}}^+,{{\bf u}}^-]$ is an edge of
$P[{{\bf u}}]$ it follows that
 ${{\bf u}}^+$, ${{\bf u}}^-$  are vertices of $P[{{\bf u}}] $
 and thus of their   Markov polytopes as well.

For the converse assume ${{\bf u}}^+$, ${{\bf u}}^-$  are vertices of
the disjoint Markov polytopes   $P_1,P_2$ respectively. Since
${{\bf u}}^+,{{\bf u}}^-$ are vertices of $P_1,P_2$ we can find
vectors ${\bf c}_1,{\bf c}_2$ such that $\supp({\bf
c}_i)\subset\supp(P_i)$ for $i=1,2$  with the property that ${\bf
c}_1\cdot {{\bf u}}^+=0$, ${\bf c}_1\cdot {{\bf v}}>0$  for all ${\bf
v}\in P_1\setminus\{{{\bf u}}^+\}$ and ${\bf c}_2\cdot {{\bf u}}^-=0$,
${\bf c}_2\cdot {{\bf v}}>0$  for all ${{\bf v}}\in P_2\setminus\{{\bf
u}^-\}$. We define ${\bf c}$  as follows
\[
{\bf c}_i=\left\{
\begin{array}{ll}
({\bf c}_1)_i, &\text{if $i\in\supp({\bf c}_1)$},\\ ({\bf c}_2)_i, & \text{if $i\in\supp({\bf c}_2)$},\\ 1, & \text{otherwise}.\\
\end{array}
\right.
\]
 From the definition of ${\bf c}$
it follows  that
${\bf c}\cdot{{\bf v}}=0$ for all ${{\bf v}}\in[{\bf{u}^+,{{\bf u}}^-}]$
and ${\bf c}\cdot{{\bf v}}>0$ for all
${{\bf v}}\in\con(P_1\cup P_2)\setminus\{[{\bf{u}^+,{{\bf u}}^-}]\}$.
On the other hand ${\bf c}\cdot {{\bf v}}>0$ for all
${{\bf v}}\notin\con(P_1\cup P_2)$, since $P[{{\bf u}}] \subset\RR_+^{n}$.
Therefore   $[{\bf{u}^+,{{\bf u}}^-}]$ is an edge of $P[{{\bf u}}] $  and by
Theorem~\ref{univ_grobner} we have ${{\bf u}}\in\mathcal U $, as desired.
\end{proof}

Note that if $\mathcal{L}$ is non positive and   
$\rank\mathcal L_{pure}>1$, then  $\MM$ is infinite by Theorem \ref{from_ctv} and thus $\MM\not\subset  \mathcal{U}$.  
On the other hand
 if
$\mathcal{L}= \mathcal{L}_{pure}$ and $\rank \mathcal L=1$    then
$\MM=\mathcal{U}=  \{{{{\bf w}}}\}$ where
${{{\bf w}}}$ is the generator of $\mathcal L$.
The proof of the next theorem follows immediately by these remarks
 and Lemma
\ref{markov_groebner_positive}.

\begin{Theorem}\label{Grobner}
Let $A$ be an arbitrary integer matrix.
The universal Markov basis of $A$ is a subset of the universal
Gr{\" o}bner basis of $A$ if and only if one of the following
two conditions holds
\begin{enumerate}
\item{} $\mathcal{L}$  is positive and every element of a Markov fiber is a vertex of a Markov polytope,
\item{} $\mathcal{L}=\mathcal{L}_{pure}$  and $\rank\mathcal{L}=1$.
\end{enumerate}
\end{Theorem}

We finish this section with an example which shows specific 
elements of $\MM$ not in $\mathcal{U}$.

\begin{Example}\label{gr_univ_mark}{\rm  Let \[
A=
 \left( \begin{array}{cccccccc}
2 & 2 & 2 & 2 & 3 & 3 & 3 & 3\\
4 & 0 & 4 & 0 & 3 & 3 & 3 & 3\\
4 & 0 & 0 & 4 & 3 & 3 &  3 & 3\\
2 & 2 & 2 & 2 & 6 & 0 & 6 & 0\\
2 & 2 & 2 & 2 & 6 & 0 & 0 & 6 \\
\end{array} \right). \]
Then  $\mathcal L$ is a positive lattice. One can prove that $\{{\bf
 u,v,w}\}$ is a  minimal Markov basis of $A$ where ${\bf
 u}=(1,1,-1,-1,0,0,0,0)$, ${{\bf v}}=(0,0,0,0,1,1,-1,-1)$ and
 ${{{\bf w}}}=
  (2,2,1,1,-1,-1,-1,-1)$.  The corresponding    Markov fibers
  are
$\mathcal F_{{\bf u}}=\{{{\bf u}}^+,{{\bf u}}^-\}$ , $\mathcal F_{\bf
v}=\{{{\bf v}}^+,{{\bf v}}^-\}$ and   $\mathcal F_{{{\bf w}}}=\{3{\bf
u}^+,{{{\bf w}}}^+,{{\bf u}}^++2{{\bf u}}^-,3{{\bf u}}^-,2{{\bf v}}^+,$ ${\bf
w}^-,2{{\bf v}}^-\}$. The  Markov polytopes of $\mathcal F_{{\bf u}}$
and $\mathcal{F}_{{\bf v}}$  are zero dimensional, they consist of
points. The Markov polytopes of   $\mathcal F_{{{\bf w}}}$ are one
dimensional: they are the line segments $\con(3{{\bf u}}^+,3{\bf
u}^-)$ and  $\con(2{{\bf v}}^+,2{{\bf v}}^-)$. We note that ${\bf
w}^+=2{{\bf u}}^++{{\bf u}}^-$ is in $\con(3{{\bf u}}^+,3{{\bf u}}^-)$,  ${{{\bf w}}}^-={{\bf v}}^++{{\bf v}}^-$ is in $\con(2{{\bf v}}^+,2{{\bf v}}^-)$ but they are not vertices. Thus ${{{\bf w}}}$ is not in the universal Gr{\" o}bner basis
of $A$. Note that $A$ has 12 different minimal Markov bases. Of
those bases, exactly $4$ are subsets of $\mathcal U $. Moreover
$|\MM |=14$  and $|\MM \cap \mathcal U |=6$. Moreover,
computing with 4ti2\cite{4ti2}, we get  that $\mathcal
G =\MM $ and  $|\mathcal U|=6$.}
 \end{Example}

\section{Generalized Lawrence liftings }\label{complexity}

Let $\mathcal{L}\subset \ZZ^n$ be a lattice.    We say that ${\bf
0}\neq {{\bf u}}$ is $\mathcal{L}$--{primitive} if $\QQ {{\bf u}}\cap
\mathcal{L} =\ZZ {{\bf u}}$.  Suppose that $\mathcal{L}$ is non
positive.  In \cite{CTV} it was shown that there exists an
$\mathcal{L} $--{\it primitive} element ${{\bf u}} \in\mathcal{L}
\cap\NN^n$ such that $\supp ({{\bf u}}) =\supp
\;\mathcal{L}_{pure}$, \cite[Proposition 2.7, Proposition
2.10]{CTV}. If  $\mathcal{L}=\mathcal{L}(A)$ then this element can
be extended to a minimal basis of $\mathcal{L}_{pure}$ and then to
a minimal Markov basis of $A$ of minimal cardinality by
\cite[Theorem 2.12, Theorem 4.1, Theorem 4.11]{CTV}. This is the
point of the next lemma.

\begin{Lemma}\label{in_basis_lem} If $\mathcal{L}$ is non positive,
there exists an $\mathcal{L}$--primitive element ${{\bf v}}\in \NN^n$  such
that ${{\bf v}}$ is in the universal Markov basis of $A$ and
$\supp({{\bf v}})=\supp\;{\mathcal{L}_{pure}}$.
\end{Lemma}

Let $A\in \MM_{m\times n}(\ZZ)$, $B\in \MM_{p\times n}(\ZZ)$ and
an integer $ r\ge 2$. We let \[ \mathcal{L}_r:=
\mathcal{L}(\Lambda(A,B,r)),\quad
\mathcal{L}_{A,B}:=\Ker_{\ZZ}(A)\cap\Ker_{\ZZ}(B) \ .\] We note
that $\mathcal{L}_r\subset\ZZ^{rn}$ while
$\mathcal{L}_{A,B}\subset\ZZ^n$.

\begin{Proposition}\label{kernel}
$\mathcal{L}_{A,B}$ is positive if and only if
$\mathcal{L}_r$ is positive for
any $r\geq 2$.
\end{Proposition}
\begin{proof} Let $C\in \mathcal{L}_{A,B}\cap \NN^n$. We think of the elements
of $\mathcal{L}_r$ as $r\times n$ matrices, as explained in the introduction.
 We have  that $[C\cdots C]^T\in \mathcal{L}_r\cap\NN^{rn}$. Conversely, if $[C_1\cdots C_r]^T\in \mathcal{L}_r\cap\NN^{rn}$
 then $C_1+\dots +C_r\in \mathcal{L}_{A,B}\cap \NN^n$.
\end{proof}

Suppose that $\mathcal{L}_r$ is positive. Let $W \in
\mathcal{L}_r$ and let  $\mathcal{W}$  the corresponding $r\times
n$ matrix with ${{{\bf w}}}_i$ as its $i$-th row.  We define $\sigma (
\mathcal{W} )=\{i:\  {{{\bf w}}}_i\neq 0, \ 1\leq i\leq r\}$. Thus the
type of $\mathcal{W}$ is the cardinality of $\sigma(\mathcal{W})$.
The
 $\Lambda(A,B,r)$--{\em degree} of $W$ is the vector $\Lambda(A,
B, r){W}^+$. Thus the $\Lambda(A, B, r)$--degree of $W$ is in the
span ${\NN}({\bf a}_i\tensor {\bf e}_j\oplus {\bf b}_i: \ 1\leq
i\leq n, j\in \sigma(\mathcal{W})   )$. It is well known that the
$\Lambda(A,B,r)$--degrees   of any minimal Markov basis of
$\Lambda(A,B,r)$  are invariants of
$\Lambda(A,B,r)$, see \cite{St}.

\begin{Theorem}\label{same_complexity} When $\mathcal{L}_r$ is positive  the complexity of a minimal Markov basis of $\Lambda(A,B,r)$  is an invariant of $\Lambda(A,B,r)$.
\end{Theorem}
\begin{proof} Let $M_1$, $M_2$ be two minimal Markov bases of $I_{\mathcal{L}_r}$.
It is enough to show that the complexity of $M_1$ is less than or
equal to the complexity of $M_2$. Let $\mathcal{W}=[{{{\bf w}}}_1
\cdots\ {{{\bf w}}}_r]^T \in M_1$ be such that the  type of
$\mathcal{W}$ is equal to the complexity of $M_1$. We let
$\mathcal{V}= [{{\bf v}}_1 \cdots\ {{\bf v}}_r]^T\in M_2$ be such that
 the $\Lambda(A,
B, r)$-degree of $V$ is the same as the $\Lambda(A, B, r)$-degree
of $W$.  Thus
 the $\Lambda(A,
B, r)$--degree of $V$ is in  ${\NN}({\bf a_i}\tensor {\bf
e_j}\oplus {\bf b_i}:\ 1\leq i\leq n,  j\in \sigma(\mathcal{W})
)$. This implies that $  {{\bf v}}^+_i=0$ for every $i\not\in
\sigma(\mathcal{W})$. Since every nonzero element in
$\Ker_{\ZZ}(A)$
 has a nonzero positive part (and a nonzero negative part) it follows that
 ${{\bf v}}_i=0$ for
every $i\not\in \sigma(\mathcal{W})$. Thus
$\sigma(\mathcal{V})\subset \sigma(\mathcal{W})$. Reversing the
argument we get that $\sigma(\mathcal{W})=\sigma(\mathcal{V})$.
Therefore the complexity of $M_1$ is less than or equal to the
complexity of $M_2$.
\end{proof}

As in \cite[Theorem 3.5]{HS} one can prove the following statement  for arbitrary integer matrices $A\in \MM_{m\times n}(\ZZ)$, $B\in \MM_{p\times n}(\ZZ)$. We denote by $\mathcal Gr(A)$ the matrix whose columns are the vectors of the Graver basis of $A$.

\begin{Theorem}\label{fin_graver_comp}
The Graver complexity $g(A,B)$ is the maximum $1$--norm of any element in the Graver basis $\mathcal{G}(B\cdot\mathcal Gr(A))$. In particular, we have $g(A,B)<\infty$.
\end{Theorem}

Suppose that $\mathcal{L}_r$ is non positive. Next we show that $\Lambda(A,B,r)$ has  a minimal Markov basis (of minimal cardinality) whose  complexity is $r$.
\begin{Theorem}\label{prim_lem}
Suppose that  $\mathcal{L}_r$ is non positive. There exists a minimal Markov basis of $\Lambda(A,B,r)$ of minimal cardinality, that contains an element of type $r$.
\end{Theorem}
\begin{proof} We first show that
$\mathcal{L}_r\cap\NN^{rn}$ has an element of type $r$.
By Lemma~\ref{kernel},  $\mathcal{L}_{A,B}$ is non positive.
We let ${{{\bf w}}}\in\mathcal{L}_{A,B}\cap\NN^n$ be  such that $\supp({{{\bf w}}})=\supp((\mathcal{L}_{A,B})_{pure})$.   It follows that
$$  \begin{pmatrix}  {{{\bf w}}}\cr\vdots\cr  {{{\bf w}}}\end{pmatrix}\in \mathcal{L}_r\cap\NN^{rn}$$
has type $r$.  Since $(\mathcal{L}_r)_{pure}=\langle \mathcal{L}_r\cap\NN^{rn}\rangle$, we are done
by Lemma~\ref{in_basis_lem}.
\end{proof}

\begin{Remark}\label{never_pure}
{\em Suppose that $r\ge 2$ and $\mathcal{L}_r$ is non positive. Let ${\bf v}$ be
  the $\mathcal{L}_r$-primitive element of Lemma~\ref{in_basis_lem}. By  adding positive multiples of ${\bf v}$
to the other elements of the Markov basis of Lemma~\ref{in_basis_lem} the new set is still  a minimal Markov basis of
$\Lambda(A,B,r)$ with the property  that all of its elements are of type $r$ (see \cite{CTV}).  
} \end{Remark}

In the next example we give matrices $A, B$ so that  for any $r\ge 2$ $\mathcal{L}(\Lambda(A,B,r))$ is non positive 
and has the following interesting property: it possesses   
 minimal Markov bases of   
complexity ranging from 2 to $r$.

\begin{Example}\label{final_example}{\rm
We let $A_1\in \mathcal{M}_{2}(\ZZ)$, $A_2\in
\mathcal{M}_{2}(\ZZ)$, $ B_2\in \mathcal{M}_{2}(\ZZ)$ and $A,
B\in\mathcal{M}_{2\times 4}(\ZZ)$ be the following matrices:
\[
A_1=\left( \begin{array}{cc}
1  & 1  \\
0  & 0
\end{array} \right), \
A_2=\left( \begin{array}{cc}
0 &0  \\
1 & -1
\end{array} \right), \
 A=\left( A_1|A_2 \right),
  B_2=\left(
\begin{array}{cc}
1 & -1 \\
 0 & 0
\end{array} \right),\]  $B=\left(I_2|B_2 \right)\ .$

We consider the matrix $ \Lambda(A,B,r)$. After column permutations it follows that $\Lambda(A,B,r)=$   
\[ 
\left( \begin{array}{cccc}
A & 0 &   & 0 \\
0 & A &  & 0 \\
 & & \ddots &  \\
0 & 0 &  & A \\
B & B & \cdots & B
\end{array} \right) \rightarrow
 \left( \begin{array}{cccc| cccc}
A_1 & 0 &  & 0 &\  A_2 & 0 &  & 0 \\
0 & A_1 & & 0 &\  0 & A_2 & & 0 \\
 & & \ddots & & & & \ddots & \\
0 & 0 &  & A_1 &\  0 & 0 & & A_2\\
I_2 & I_2 & \cdots & I_2 & \ B_2 & B_2 & \cdots & B_2
\end{array} \right).\]
We note that the lattice $\mathcal{L}\left(
\Lambda(A_1,I_2,r)|\Lambda(A_2,B_2,r)\right)$ is isomorphic to the
direct sum of the lattices $\mathcal{L} (\Lambda(A_1,I_2,r))$ and
$\mathcal{L}(\Lambda(A_2,B_2,r))$.

The matrix  $\Lambda(A_1,I_2,r)$
 is the
defining matrix of the toric ideal of the complete bipartite graph
 $K_{2, r}$ and has a unique minimal Markov basis corresponding to cycles of length $4$:
 all its elements have type $2$, see \cite{OH} and
 \cite{V}. We denote by $C_i$ the columns of  $ \Lambda(A_2,B_2,r)$, for $i=1,\ldots, 2r$. We
note that $C_1, C_3, \ldots, C_{2r-1}$  are linearly independent while $C_{2l-1}=-C_{2l}$ for
 $1\leq l\leq r$.  It follows that the lattice $\mathcal{L} (\Lambda(A_2,B_2,r))$ has rank $r$ and is pure. Thus  
 it has infinitely many  Markov bases ( see \cite{CTV}). 
 We consider the following minimal Markov basis of $\Lambda(A_2,B_2,r)$ consisting of elements of type 1:
$$\{\begin{array}{ cccc}\left( \begin{array}{cc}
1 & 1\\
0 & 0\\
\vdots & \vdots\\
0 & 0
\end{array} \right),& \left( \begin{array}{cc}
0 & 0\\
1 & 1\\
\vdots & \vdots\\
0 & 0
\end{array} \right),& \ldots, &\left( \begin{array}{cc}
0 & 0\\
0 & 0\\
\vdots & \vdots\\
1 & 1
\end{array} \right)\end{array}\}.$$

For fixed $1\le a\le r$ and $1\le b\le 4$  we let $E_{a,b}$ be the
matrix of $\mathcal{M}_{r\times 4}(\ZZ)$  which has $1$ on the
$(a,b)$-th entry and $0$ everywhere else. Moreover for $1\leq
i<j\leq r$ and  $1\leq s\leq r$, we let
$P_{i,j}\in\mathcal{M}_{r\times 4}(\ZZ)$ and
$T_s\in\mathcal{M}_{r\times 4}(\ZZ)$
 be the matrices
\[
P_{ i,j }=E_{i,1}-E_{i,2}-E_{j,1}+E_{j,2},\quad
T_s=E_{s,3}+E_{s,4} \ .\]

 It follows that  the set $\mathcal{M}=\{T_1,\ldots,T_r\}\cup\{P_{i,j }:\
1\leq i<j\leq r \}$ is a minimal Markov basis of $\Lambda(A,B,r)$
of cardinality $r+\binom{r}{2}$. The
 elements of $\mathcal{M}$ have type $1$ and $2$, therefore the complexity of this Markov basis is $2$.

Note  that the  set
\[
\{T_1,T_1+T_2,\ldots,T_1+\cdots+T_k, T_{k+1}, \cdots, T_r\}\cup\{P_{i,j }:\ 1\leq
i<j\leq r \}
\]
is a minimal Markov basis of $\Lambda(A,B,r)$ and the type of its elements range from $1$ to $k$, where $2\leq k\leq r$. Therefore for any integer $k$ 
between $2$ and $r$ there are
minimal Markov bases  of $\Lambda(A,B,r)$ of complexity $k$.

\noindent Moreover if $T=\sum_{s=1}^r T_s$, then the set
\[
\{T,T+T_2,\ldots,T+T_r\}\cup\{T+P_{i,j }:\ 1\leq i<j\leq r \}
\]
is a minimal Markov basis of $\Lambda(A,B,r)$ such that all its elements are of
type $r$ (see \cite{CTV}).

We remark that if $S$ is any integer linear combination of the
elements $T_s$, $1\leq s\leq r$ and $1\leq i<j\leq r$ then
 the element
$S+P_{ i,j }$ belongs to the infinite universal Markov basis of
$\Lambda(A,B,r)$.

}\end{Example}

\begin{Remark}\label{markov-complexity}
{\em As pointed out in the introduction, in the literature there are two definitions of  Markov complexity. The  one introduced  in \cite{SS}, namely 
 the smallest integer $m$ such that there exists a Markov basis  of
$\Lambda (A,B,r)$ of type less than or equal to $m$ for any $r\ge 2$, is always finite:  
there exists a minimal Markov basis inside the Graver basis
and thus (this) Markov complexity  is always smaller than the Graver complexity. When 
$\mathcal{L}_{A,B}$ is a positive lattice then Theorem 
\ref{same_complexity} becomes 
essential in the computation of this Markov complexity:  it guarantees that all minimal Markov bases have the same complexity.
Even when  
$\mathcal{L}_{A,B}$ is a non positive lattice this  Markov complexity  
is finite. For example, the
Markov complexity of  $(A,B)$ of
Example \ref{final_example}    is
equal to 2.
The second definition was given in \cite{HS} where the Markov complexity of  $(A,B)$
 is the largest type of any element in the universal Markov basis of $\Lambda (A,B,r)$ as $r$ varies.
 It is clear form Theorem~\ref{prim_lem} that this Markov complexity  
is infinite if  $\mathcal{L}_{A,B}$ is a non positive lattice. 
 We  
point out that when  $\mathcal{L}_{A,B}$ is a positive lattice
which is the main case of interest in Algebraic Statistics,
by  Theorem~\ref{same_complexity},
all minimal Markov bases of $\Lambda(A, B, r)$ have the same  
complexity for   $r\geq 2$ and thus
the definition of \cite{HS} agrees computationally with the definition  
of \cite{SS}.
} \end{Remark}

{\bf Acknowledgment}. This paper was partially written during the visit of the first and third author at the University of Ioannina. 
The third author was supported by a Romanian grant awarded by UEFISCDI, project number $83/2010$, PNII-RU code TE$\_46/2010$, 
program Human Resources, ``Algebraic modeling of some combinatorial objects and computational applications''.

\end{document}